\documentclass[12pt,final]{amsproc}
\usepackage{amssymb,amsmath,amsthm,latexsym, amscd}
\usepackage{mathrsfs}
\usepackage{a4wide}
\usepackage[usenames,dvipsnames]{color}
\usepackage[utf8]{inputenc}
\usepackage[T1]{fontenc}
\usepackage{dsfont}
\usepackage{tikz}
\usepackage{float}
\usepackage[mathscr]{euscript}
\usepackage{mathtools}
\usepackage{hyperref}

\usepackage[all]{xy}
\usepackage{mathptmx}

\newcommand{\R}{\mathbb R}
\newcommand{\N}{\mathbb N}
\newcommand{\e}{\varepsilon}

\newcommand{\Aa}{{\mathcal A}}

\newcommand{\U}{\mathfrak{U}}  

\newcommand{\dist}{\mathop{\rm dist}}

\theoremstyle{plain}
\newtheorem{theorem}{Theorem}
\newtheorem{prop}{Proposition}
\newtheorem{lemma}{Lemma}
\newtheorem{cor}{Corollary}
\theoremstyle{definition}
\newtheorem{definition}[theorem]{Definition}
\theoremstyle{remark}
\newtheorem{remark}{Remark}
\newtheorem{example}{Example}

\def\PB{\operatorname{PB}}

\begin{document}

\title{The isomorphic Kottman constant of a Banach space}

\author[J.M.F.~Castillo]{Jes\'us M.~F.~Castillo}
\address{Instituto de Matem\'aticas de la Universidad de Extremadura (IMUEX),
Avda de Elvas s/n, 06011 Badajoz, Spain.}
\email{castillo@unex.es}

\author[M. Gonz\'alez]{Manuel Gonz\'alez}
\address{Departamento de Matem\'aticas, Universidad de Cantabria,
Avda de los Castros s/n, E-39071 Santander, Spain.}
\email{manuel.gonzalez@unican.es}

\author[T.~Kania]{Tomasz Kania}
\address{Institute of Mathematics, Czech Academy of Sciences, \v{Z}itn\'{a} 25, 115~67 Prague 1,
Czech Republic, and
Institute of Mathematics, Jagiellonian University, {\L}ojasiewicza 6, 30-348 Kra\-k\'{o}w, Poland}
\email{kania@math.cas.cz, tomasz.marcin.kania@gmail.com}

\author[P.~Papini]{Pier Luigi Papini}
\address{via Martucci 19, 40136 Bologna, Italia} \email{pierluigi.papini@unibo.it}

\subjclass[2010]{46B03, 46B08, 46B10}
\keywords{Kottman constant, Banach space, twisted sum, separated set}
\thanks{The research of the first-named author was supported in part by Project IB16056 de la Junta de Extremadura,
and the first- and second-named authors were supported in part by Project MTM2016-76958 (Spain).
The third-named author acknowledges with thanks funding received from GA\v{C}R project 19-07129Y; RVO 67985840.}

\begin{abstract}
We show that the Kottman constant $K(\cdot)$, together with its symmetric and finite variations, is continuous with respect 
to the Kadets metric, and they are log-convex, hence continuous, with respect to the interpolation parameter in a complex interpolation schema. 
Moreover, we show that $K(X)\cdot K(X^*)\geqslant 2$ for every infinite-dimensional Banach space $X$. 

We also consider the isomorphic Kottman constant (defined as the infimum of the Kottman constants taken over all renormings of the space) and solve the main problem left open in \cite{cgp}, namely that the isomorphic Kottman constant of a twisted-sum space is the maximum of the constants of the respective summands.
Consequently, the Kalton--Peck space may be renormed to have Kottman's constant arbitrarily close to $\sqrt{2}$. 
For other classical parameters, such as the Whitley and the James constants, we prove the continuity with
respect to the Kadets metric.
\end{abstract}

\maketitle

\section{Introduction}
We continue the study of the separation of sequences in the unit ball $B_X$ of an infinite-dimensional Banach space $X$, solving a few problems left open in \cite{cgp,castpapi,H-kania-R} concerning the Kottman
constant of $X$ and variations thereof.
We refer to the above-mentioned papers for the relevant background. Before we describe our main results, we gather some relevant definitions and facts.\smallskip

Unless otherwise specified, we tacitly assume that \emph{a space} is an infinite-dimensional Banach space.
The \emph{Kottman constant} of a space $X$, denoted $K(X)$, is defined as
$$
K(X) = \sup \{ \sigma >0\colon \exists (x_n)_{n=1}^\infty\text{ in } B_X \text{ such that }
\|x_n -x_m\|\geqslant \sigma \text{ for } n\neq m\}
$$
and is accompanied by its variations:
$$
K_s(X)=\sup \{ \sigma >0\colon \exists (x_n)_{n=1}^\infty\text{ in } B_X \text{ such that }
\|x_n \pm x_m\|\geqslant \sigma \text{ for } n\neq m\},
$$
$$
K_f(X) = \sup \{ \sigma >0\colon \forall N\in \mathbb N\; \exists  (x_n)_{n=1}^N\text{ in }B_X
 \text{ such that } \|x_n - x_m\|\geqslant \sigma \text{ for } n\neq m\},
$$
called, respectively, the \emph{symmetric} and \emph{finite} Kottman constants.\medskip

Next we list some relevant facts concerning these constants:

\begin{itemize}
\item  \cite{castpapi,eltod} For a countably incomplete ultrafilter $\U$ (in particular, for any non-principal ultrafilter on a countable set) and a space $X$, we have
$$
1<K(X)\leqslant K_f(X)=K(X_\U)\leqslant 2,
$$
where $X_\U$ stands for the ultrapower of $X$ with respect to $\U$.
\item \cite[Proposition 5.1]{H-kania-R}, \cite{kott,vandulst} Every space $X$ may be renormed so that
$$
K_s(X)=2 \textrm{ or } K(X)=K(X^*)=2.
$$
\item \cite{cgp} There exists a space $Z$ for which $K(Z)<K(Z^{**})$, and it is easy to check that this space
also satisfies $K_s(Z)<K_s(Z^{**})$. The said space is a $J$-sum of $\ell_1^n$ ($n\in \mathbb N$) in the sense of Bellenot (\cite{bellenot}); it has the property that $K(Z)<2$, yet $Z^{**}$ admits a quotient map onto $\ell_1$ so that $K_s(Z^{**})=2$.
\end{itemize}
The fact that $K(X)>1$ is known as the Elton--Odell theorem \cite{eltod}.
Kottman had previously shown \cite{kott2} that $K(X)>1^+$, meaning that there is a sequence $(x_n)_{n=1}^\infty$
in $B_X$ such that $\|x_n -x_m\|>1$ for distinct natural numbers $n,m$.
In \cite{H-kania-R} it was proved that $K_s(X)>1^+$ and $K_s(X)>1$ for every separable dual space $X$, and
recently Russo proved that $K_s(X)>1$ for every $X$ \cite{russo}.\medskip

In this paper, among other things, we study the interrelation between the Kottman constants with interpolation spaces and twisted sums of Banach
spaces, proving the following facts:
\begin{enumerate}
\item The inequality $2\leqslant K(X)\cdot K(X^*)$ is valid for any space $X$.

\item The above-listed Kottman constants are continuous with respect to the Kadets metric, which
implies their continuity with respect to the interpolation parameter.
Moreover, under some additional conditions, the following interpolation inequality is established:
$$K(X_\theta)\leqslant K(X_0)^{1-\theta}\cdot K(X_1)^\theta.$$

\item The isomorphic Kottman constant $\tilde K(X) = \inf \{K(\tilde X)\colon \tilde{X}\cong X\}$ that was
introduced in \cite{cgp} to treat some natural situations in which no specific norm of a space is known, is computed for twisted sums in terms of the isomorphic constants of the summands. More specifically, for a twisted sum $X$ expressed in terms of the short exact sequence $0 \to Y \to X\to  Z\to 0$, the formula 
$$\tilde K(X)= \max \{\tilde K(Y), \tilde K(Z)\}$$
is established, which solves a problem posed in \cite{cgp}. In particular, if $X$ is a twisted Hilbert space, namely a space that can be represented as a twisted sum of two Hilbert spaces, then $\tilde K(X)=\sqrt{2}$.\smallskip

\item For the \emph{disjoint Kottman constant} $K^\perp$ of K\"othe spaces, that we introduce here, we prove
some results, including a general interpolation formula:
$$K^\perp(X_\theta)\leqslant K^\perp(X_0)^{1-\theta}\cdot K^\perp(X_1)^\theta.$$
\end{enumerate}

The results presented above are also valid for both the symmetric and finite Kottman constants as well as for their isomorphic variations.\smallskip

The final section of the paper is devoted to linking and extending this study to other well-known  parameters such as the Whitley thickness constant \cite{castpapiH} and the James constant \cite{castpapi}; 
a~number of applications to the geometry of Banach spaces is presented.

\section{Estimates for the Kottman constant, continuity, and interpolation}
\subsection{A relation between the constants of a space and its dual}

Our first lemma is apparently a folklore result, however we have been unable to identify a proper reference in the literature, so we include a proof for the sake of completeness.

Let $A$ be an infinite subset of $\N$ for which we set
$[A]_2=\{(n_1,n_2)\in A\times A : n_1<n_2\}$.
Ramsey's theorem \cite[Theorem 1.1]{odell} asserts that given $\Aa\subset [\N]_2$, there exists an infinite subset $B$ of $\N$ such that either $[B]_2\subset \Aa$ or $[B]_2\subset [\N]_2\setminus\Aa$.

\begin{lemma}\label{ramsey}
Let $(x_n)$ be a bounded sequence in a Banach space.
Then there exists an infinite subset $M$ of $\N$ such that the sequence $\|x_i-x_j\|$ converges as $i,j\in M$, $i,j\to\infty$.
\end{lemma}
\begin{proof}
We may suppose that $\{\|x_i-x_j\|\colon i,j\in \N, i<j\}$ is contained in an interval $[a,b]$.
Let $c=(a+b)/2$ be the midpoint and let $\Aa=\{(n_i,n_j)\in [\N]_2: \|x_{n_i}-x_{n_j}\|\in [a,c]\}$. By Ramsey's theorem there exists an infinite subset $M_1$ of $\N$ such that
$\{\|x_i-x_j\|\colon (i,j)\in [M_1]_2\}$ is contained in $[a,c]$ or in $(c,b]$.\smallskip

Repeating the process, we obtain a decreasing sequence $M_1\supset M_2\supset\cdots$ of
infinite subsets of $\N$ such that the set $\{\|x_i-x_j\| : (i,j)\in [M_k]_2\}$ has
diameter at most $(b-a)/2^k$. Then the set $M=\{m_1<m_2<\cdots\}\subset\N$ with $m_k\in M_k$ meets the requirements and witnesses the convergence of $(\|x_i-x_j\|)_{i,j\in M}$ as $i,j\to\infty$.
\end{proof}

\begin{prop}
For every infinite-dimensional Banach space $X$ we have $2\leqslant K(X)\cdot K(X^*)$.
\end{prop}
\begin{proof}
Day proved in \cite{day} that the unit sphere of $X$ contains a basic sequence $(x_n)_{n=1}^\infty$ whose coordinate functionals have norm one. Let $(x^*_n)_{n=1}^\infty$ denote the sequence of arbitrary norm-one extensions of the coordinate functionals associated to $(x_n)_{n=1}^\infty$.
By biorthogonality of $(x_n, x_n^*)_{n=1}^\infty$, we have
$$
2=\langle x^*_i-x^*_j, x_i-x_j\rangle \leqslant \|x^*_i-x^*_j\|\cdot\|x_i-x_j\|.
$$
Passing to a subsequence if necessary, we may assume that both $\|x^*_i-x^*_j\|$ and $\|x_i-x_j\|$ converge
in the sense of Lemma \ref{ramsey} to $k^*$ and to $k$, respectively.
Then $2\leqslant k^*\cdot k\leqslant K(X^*) \cdot K(X)$.
\end{proof}

\subsection{Continuity of the Kottman constant and interpolation inequalities}

The Kottman constant is readily continuous with respect to the Banach--Mazur distance \cite{kott},
with a simple estimate $K(X)\leqslant K(Y)\cdot d_{\rm BM}(X,Y)^2$.
In particular, two Banach spaces with the Banach--Mazur distance equal to 1 have the same Kottman constant.
We are however interested in continuity with respect to the so-called Kadets distance.\smallskip

Let $M, N$ be closed subspaces of a Banach space $Z$.
The \emph{gap $g(M,N)$ between $M$ and $N$} is defined as
$$
g(M,N) = \max \big\{\sup_{x\in B_M}\dist(x,B_N), \sup_{y\in B_N}\dist(y,B_M)\big\},
$$
where $\dist(x,B_N)=\inf\{\|x-n\|\colon n\in B_N\}$. The \emph{Kadets distance} $d_{\rm K}$ between two Banach spaces $X,Y$
is defined as the infimum of $g(iX, jY)$, where $i\colon X\to W$, $j\colon Y\to W$ range through isometric embeddings into the same Banach space $W$.
We are ready to present the following elementary result concerning continuity of the Kottman constant with respect to $d_K$.

\begin{theorem}\label{kottcontka}
The Kottman constant is continuous with respect to the Kadets metric. More precisely,
$$|K(X) - K(Y)|\leqslant 2 \cdot d_{\rm K}(X, Y).$$
The same is true for both symmetric and finite Kottman constants.
\end{theorem}
\begin{proof}
Certainly, for isometric embeddings $i,j$, we have $K(X)=K(iX)$ and $K(Y)=K(jY)$. This together with Lemma \ref{contg} below yield
$|K(iX) - K(jY)|\leqslant 2 g(iX, jY)$ and, consequently, $|K(iX) - K(jY)|\leqslant 2 d_{\rm K}(X, Y)$.
It is clear that the result is also valid for $K_s(\cdot)$ and $K_f(\cdot)$.
\end{proof}

\begin{lemma}\label{contg}
Let $M,N$ be subspaces of a Banach space $Z$. Then $|K(M) - K(N)|\leqslant 2\cdot g(M,N)$.
\end{lemma}
\begin{proof}
We will present the proof only for $K$ as for $K_s$ it will be entirely analogous.\smallskip

Let us assume for the sake of simplicity that $K(M)$ is attained. So we may find a sequence $(a_n)_{n=1}^\infty$ in $B_M$ such that $K(M) = \|a_n - a_m\|$.
For each $a_n$ ($n\in \N$) we pick some $b_n$ in $B_N$ so that $\|a_n - b_n\|\leqslant g(M,N)$. Then
$$\|b_n - b_m\| \geqslant  K(M) - 2\cdot g(M,N).$$

Consequently, $K(N)\geqslant  K(M) - 2\cdot g(M,N)$, hence $K(M) - K(N)\leqslant 2\cdot g(M,N)$, and exchanging the
r\^{o}les of $M$ and $N$ one finally gets $\left |K(N) - K(M)\right|\leqslant 2 \cdot g(M,N)$.
\end{proof}

\subsection{Complex interpolation and separation}
We refer the reader to \cite{belo} for all necessary information on complex interpolation theory for
Banach spaces.\smallskip

Let $(X_0, X_1)$ be an interpolation couple, let $\mathbb S=\{z\in \mathbb C\colon 0< {\rm Re}\, z < 1\}$ be the complex unit strip, and let $\mathscr C= \mathscr{C}(X_0, X_1)$ be the \emph{Calderon space} formed by those
bounded continuous functions $F\colon\overline{\mathbb{S}} \to X_0+X_1$ which are analytic on $\mathbb{S}$, satisfy
the boundary conditions $F(k+ti)\in X_k$ for $k=0,1$, and the norm
$\|F\|_{\mathscr C}=\sup\{\|F(k+ti)\|_{X_k}: t\in\R,\; k=0,1\}$ is finite.\smallskip

For each $\theta\in \mathbb S$ we may consider the evaluation functional $\delta_\theta\colon \mathscr C\to X_0+X_1$, which is defined by
$\delta_\theta(f)=f(\theta)$.
The interpolation spaces are quotient spaces $X_\theta = (X_0,X_1)_\theta = \mathscr C / \ker\delta_\theta$ endowed with their
natural quotient norm. Kalton and Ostrovskii \cite{kaltostr} proved that the Kadets metric is continuous with
respect to the interpolation parameter, by showing that
$$
d_{\rm K}(X_t, X_s)\leqslant 2 \left| \frac{\sin \left(\pi(t-s)/2\right)}{\sin \left(\pi(t+s)/2\right)}\right|,\quad 0<s,t<1.
$$

Thus, by combining the continuity of Kottmant's constant with respect to the Kadets distance together with the
continuity of the Kadets metric with respect to the interpolation parameter yields the following corollary.

\begin{cor}\label{cor:cont}
Let $(X_0, X_1)$ be an interpolation couple. Then the (symmetric, finite) Kottman constant is continuous with
respect to the interpolation parameter; precisely
$$
|K(X_t) - K(X_s)|\leqslant 4 \left| \frac{\sin \left(\pi(t-s)/2\right)}{\sin \left(\pi(t+s)/2\right)}\right|,\quad 0<s,t<1.
$$
\end{cor}


Next, we improve Corollary \ref{cor:cont} by establishing log-convexity of the interpolation inequalities, that is, that they are of the form
$K(X_\theta) \leqslant K(X_0)^{1-\theta}\cdot K(X_1)^\theta$.
To do that we need an equivalent description of the complex interpolation method given in \cite{daher} which
we briefly explain in the subsequent paragraphs.\smallskip

We denote by $\overline{X}$ the interpolation couple $(X_0,X_1)$, and for $j=0,1$, $z=s+it\in \mathbb S$ and
$\tau\in\R$ we set $d\mu_{z,j}(t)=Q_j(z,t)dt$ (see \cite{daher}), and for $1\leqslant p<\infty$ and $0<\theta<1$, we denote by
$\mathcal{F}^p_\theta(\overline{X})$ the space of functions $F\colon \overline{\mathbb S}\to X_0+X_1$ such that $F$
is analytic on $\mathbb S$, the functions $F_j(\tau)=F(j+i\tau)$ are Bochner-measurable with values in $X_j$ and satisfy
\begin{eqnarray}\label{eq:norm} 
\|F\|_{\mathcal{F}^p_\theta(\overline{X})}=
\int_\R \|F(it)\|_0^p\, \mu_{\theta,0}({\rm d}t)+\int_\R \|F(1+it)\|_1^p\, \mu_{\theta,1}({\rm d}t)<\infty.
\end{eqnarray}

For $p=\infty$ we similarly define $\mathcal{F}^\infty(\overline{X})$, independent of $\theta$,
replacing condition (\ref{eq:norm}) by
$$
\|F\|_{\mathcal{F}^\infty(\overline{X})}= \max\{\sup_{t\in\R}\|F(it)\|_0,\; \sup_{t\in\R}\|F(1+it)\|_1\}<\infty.
$$

Let us observe that $\mu_{\theta,0}$ and $\mu_{\theta,1}$ are finite measures on $\R$.
Therefore we have the inclusion $\mathcal{F}^\infty(\overline{X})\subset \mathcal{F}^p(\overline{X})$ for
$1\leqslant p<\infty$.
It was proved in \cite{daher} that $X_\theta = \{F(\theta) \colon F\in \mathcal{F}^\infty(\overline{X})\}$ and
$\|x\|_\theta = \inf\{\|F\|_{\mathcal{F}^\infty(\overline{X})} \colon F(\theta)=x\}$.
\smallskip

An interpolation couple $(X_0, X_1)$ is called \emph{regular}, whenever $X_0\cap X_1$ is dense in both $X_0$ and $X_1$.
Given $\theta\in \mathbb S$ and $x\in X_\theta$, an element $f\in \mathcal{F}^\infty(\overline{X})$ is called a
\emph{$1$-extremal for $x$ at $\theta$} if $f(\theta)=x$ and $\|f\|_{\mathcal{F}^\infty(\overline{X})}=\|x\|_\theta$.
We require the following technical result, whose proof is contained in \cite[Th\'eor\`eme]{daher}.
We include some details of the proof for completeness.

\begin{lemma}\label{lemma:extremal}
Let $(X_0, X_1)$ be a regular interpolation pair of reflexive spaces.
Given $x \in X_0\cap X_1$ and $\theta \in (0,1)$ there exists a $1$-extremal $f_{x,\theta}$ for $x$
at $\theta$ such that $\|f_{x, \theta}(z)\|_z = \|x\|_z$ for every $z\in \mathbb{S}$.
\end{lemma}
\begin{proof}
Suppose $\|x\|_\theta=\|f_{x,\theta}\|_{\mathcal{F}^\infty(\overline{X})}=1$.
%
We select $x^*\in X_\theta^*$ such that $\|x^*\|=\langle x,x^*\rangle=1$.
By \cite[part I in Proposition 3]{daher}, there exists $f^*\in \mathcal{F}^2(\overline{X^*})$ with
$f^*(\theta)=x^*$ and $\|f^*\|_{\mathscr{F}^2(\overline{X^*})}=1$.
Applying \cite[Lemma 4.2.3]{belo}, we can show that $g(z)=\langle f_{x, \theta}(z),f^*(z)\rangle$
defines an analytic function.
Since $|g(z)|\leqslant 1$ for every $z \in \mathbb{S}$ and $g(\theta)=1$, the maximum principle
for analytic functions implies that $g(z)= 1$ for every $z \in \mathbb{S}$.
%
Therefore $\|f_{x, \theta}(z)\|_z = 1$ for every $z \in \mathbb{S}$.
\end{proof}


\begin{theorem}\label{interpol}
Let $(X_0, X_1)$ be regular interpolation pair of Banach spaces with $X_0$ reflexive and let $0<a<b<1$.
Then
$$K(X_{(1-\theta)a+\theta b}) \leqslant K(X_a)^{1-\theta} K(X_b)^\theta\quad \big(\theta\in (0,1)\big).$$
The inequality is valid for $K_s(\cdot)$ and $K_f(\cdot)$ as well.
\end{theorem}
\begin{proof}
Denoting $\gamma =(1-\theta)a+\theta b$, we have $\|x\|_\gamma\leqslant \|x\|_a^{1-\theta}\|x\|_b^\theta$ for each $x\in X_a\cap X_b$.

Let $\varepsilon > 0$. We pick an almost optimal Kottman sequence in $X_\gamma$, that is, a sequence $(x_n)_{n=1}^\infty$
such that $\|x_n\|_\gamma=1$ and $K(X_\gamma) - \e\leqslant \|x_n - x_m\|\leqslant K(X_\gamma) + \e$ for $n\neq m$.
Since the interpolation pair is regular, we can assume $(x_n)_{n=1}^\infty\subset X_0\cap X_1\subset X_a\cap X_b$.
For each $n$ we take the $1$-extremal $f_{n,\gamma}$ for $x_n$ at $\gamma$, given by Lemma \ref{lemma:extremal}.
Then $\|f_{n,\gamma}\|_{\mathcal{F}^\infty(\overline{X})}= \|f_{n,\gamma}(\gamma)\|_\gamma=1$ and
\begin{eqnarray*}\label{eq:K-interp}
K(X_\gamma)-\e &\leqslant &   \|x_n - x_m\|_\gamma \\
&=&  \|f_{n,\gamma}(\gamma) - f_{m,\gamma}(\gamma)\|_\gamma\\
&\leqslant& \sup_{t\in\R} \|f_{n,\gamma}(a+it) - f_{m,\gamma}(a+it)\|_{a+it}^{1-\theta}\;
\sup_{t\in\R} \|f_{n,\gamma}(b+it) - f_{m,\gamma}(b+it)\|_{b+it}^\theta
\end{eqnarray*}
by Hadamard's three-lines theorem \cite[Lemma 1.1.2]{belo}.

By Lemma \ref{lemma:extremal}, for each $t\in\R$ we have
$\|f_{n,\gamma}(a+it)\|_{a+it}=\|f_{n,\gamma}(\gamma)\|_\gamma=\|x_n\|_\gamma =1$, and similarly
$\|f_{m,\gamma}(a+it)\|_{a+it}=\|f_{n,\gamma}(b+it)\|_{b+it}=\|f_{n,\gamma}(b+it)\|_{b+it}=1$.
Moreover, by the invariance of the strip $\mathbb S$ under vertical translations, given $s\in (0,1)$
and $t\in \R$ we have $X_{s+it} = X_s$ with equal norms.
Thus the previous chain of inequalities gives $K(X_\gamma)-\e\leqslant K(X_a)^{1-\theta}K(X_b)^\theta$,
proving the result.

The same argument works for both the symmetric and finite Kottman constants.
\end{proof}

It would be interesting to know if Theorem \ref{interpol} is valid with $a=0$ and $b=1$.
\medskip


A forerunner of Theorem \ref{interpol} appears in \cite[Theorem 1]{appsem} in the following form:
\emph{If $0 < p < 1$ and $E$ is a $\theta$-Hilbert space, then} $K_f(E)\leqslant 2^{1-\theta/2}$.
This formula matches the $K_f$-inequality in Theorem \ref{interpol}, as indeed, $E$ is  a~$\theta$-Hilbert space according to Pisier \cite{pisi}, whenever $E=(X, H)_\theta$ for a Hilbert space $H$.
Note that we may always assume that $X$ is reflexive because $X_1$ reflexive implies reflexivity of $X_t$ for all $t\in (0,1)$. Thus, Theorem \ref{interpol} gives the following estimate:
$$
K_f(E) \leqslant K_f(X)^{1-\theta}K_f(H)^\theta \leqslant 2^{1-\theta}2^{\theta/2} =  2^{1-\theta/2}.
$$





An interesting case occurs when one considers a K\"othe space $\lambda$ of $\mu$-measurable functions
and its $p$-convexification $\lambda_p$ for $1\leqslant p<+\infty$ endowed with the norm $\|x\|_p = \| |x|^p \|^{1/p}$. For $p=\theta^{-1}$ we have $\lambda_p = (L_\infty(\mu), \lambda)_\theta$ \cite[Proposition 3.6]{cfg}. Conversely, if $X$ is $p$-convex and $X^{p}$ is the $p$-concavification of $X$, then $X=(L_\infty(\mu),X^p)_{1/p}$,
which yields $K(\lambda_p)\leqslant K(\lambda)^{1/p}2^{1/p^*}.$
\medskip


Calderon's paper \cite{calde} contains a general interpolation result for vector sums that we describe now.
Let $\lambda$ be a K\"othe space of $\mu$-measurable functions. Given a Banach space $X$ one can form the vector
valued space $\lambda(X)$ of measurable functions $f\colon S \to X$ such that the function
$\widehat f (\cdot) = \|f( \cdot)\|_X\colon S\to \R$ given by $t \to \|f(t)\|_X$ is in $\lambda$, endowed with the norm
$\| \|f(\cdot)\|_X \|_\lambda$.

\begin{prop}\label{Cinterpolation}
Fix $0<\theta<1$. Let $(\lambda_0, \lambda_1)$ be an interpolation couple of Banach function spaces on the same measure space for which $(\lambda_0, \lambda_1)_\theta = \lambda_0^{1-\theta}\lambda_1^{\theta}$, and let $(X_0, X_1)$ be an
interpolation couple of Banach spaces. Suppose that $\lambda_0(X_0)$ is reflexive. Then
$$(\lambda_0(X_0), \lambda_1(X_1))_\theta = \lambda_0^{1-\theta}\lambda_1^\theta \left ((X_0, X_1)_\theta\right).$$
\end{prop}
In general, the interpolation formula yields
\begin{eqnarray*}
K\left( (\lambda_0(X_0), \lambda_1(X_1))_\theta \right) &\leqslant& K\left( \lambda_0(X_0)\right)^{1-\theta} 
K\left( \lambda_1(X_1)\right)^{\theta}\\
&= &\max \lbrace K(\lambda_0), K(X_0)\rbrace ^{1-\theta} \max \lbrace K(\lambda_1), K(X_1)\rbrace ^{\theta} 
\end{eqnarray*}
according to \cite[Proposition 1.1]{castpapi}. However, under the conditions above one obtains the estimate
\begin{eqnarray*}
K\left( \lambda_0^{1-\theta}\lambda_1^\theta \left ((X_0, X_1)_\theta\right) \right) &=& 
\max \lbrace K\left(\lambda_0^{1-\theta}\lambda_1^\theta\right), K \left( (X_0, X_1)_{\theta}\right) \rbrace\\
&\leqslant &\max \lbrace K(\lambda_0)^{1-\theta}K(\lambda_1)^{\theta},  K(X_0)^{1-\theta}K(X_1)^{\theta} \rbrace
\end{eqnarray*}
which is, in general, better.\smallskip

The result translates verbatim to the cases of symmetric and finite Kottman constants.\medskip

\begin{remark}
The interpolation formulae for $K(\cdot)$ and $K_f(\cdot)$ are somewhat surprising.
To explain why it is, let us recall the following parameters of a (bounded, linear) operator $T\colon X\to X$ on a Banach space $X$. The \emph{outer entropy numbers} of $T$ are defined by
$$e_n(T) = \inf\Big\{ \sigma \geqslant  0: \exists y_1, \dots, y_n\colon \; T\left( B_X\right ) \subset \bigcup y_i + \sigma B_X\Big\},$$
while the \emph{inner entropy numbers} are defined by
$$f_n(T) = \sup \Big\{ \sigma \geqslant  0\colon \exists x_1, \dots, x_n\colon \; \|x_i - x_j\| \geqslant  \sigma \Big\},$$
see \cite[Chapter 12]{pietsch} for more details. \smallskip

\noindent \emph{Warning!}~Pietsch calls $f_n$ what in our case is $\frac{1}{2}f_{2^n}$ and $e_n$ for what we denote by $e_{2^n}$; this is irrelevant for our discussion, though.\smallskip

It is clear that $K_f(X) = \lim\sup f_n({\rm id}_X)$ while  $\beta(X) = \lim\inf e_n({\rm id}_X)$ is the Carl and Stephani measure of non-compactness \cite{cast}.
Pietsch presents interpolation formulae for both inner and outer entropy numbers, however only in the setting of operators with a fixed domain or
codomain, which is not the case when one consider identities.
Theorem \ref{interpol} yields that in fact
$$\lim\sup f_n({\rm id}_{X_\theta}) \leqslant  \lim\sup f_n({\rm id}_{X_0})^{1-\theta}\lim\sup f_n({\rm id}_{X_1})^\theta.$$
The case of $\beta$ is remarkable since there are interpolation formulae for $\beta$ \cite{cobos,sz}, although not for the entropy numbers \cite{edm}.
\end{remark}

\section{The isomorphic Kottman constant for twisted sums}

When a space $X$ is defined by an exact sequence $0\to Y \to X \to Z \to 0$ then it usually lacks a canonical norm, and it may have several realisations up to an isomorphism.

Probably, the best example is the Kalton--Peck $Z_2$ space \cite{kaltpeck}: this space is defined to be
a~non-trivial twisted Hilbert space; namely, there exists an exact sequence $0\to \ell_2 \to Z_2 \to \ell_2 \to 0$
that does not split and thus the space $Z_2$ cannot be isomorphic to a Hilbert space.
To construct the space $Z_2$ we require a non-trivial quasi-linear map $\Omega\colon \ell_2\to \ell_2$, actually a map given by $$\Omega(x) = x\log\left(|x_n|/\|x\|_2\right) \quad (x\in \ell_2).$$
The space $Z_2$ carries a natural quasi-norm given by $\|(y,x)\|= \|y - \Omega x\|_2 + \|x\|_2$ ($(y,x)\in Z_2$). In order to prove that it is a
Banach space one must invoke a deep result of Kalton \cite{kalt} showing that the convex hull of the unit ball of the preceding quasi-norm actually provides an equivalent topology.
In \cite{cgp} it was shown that the Kottman constant of this norm is strictly bigger than $\sqrt{2}$.
The question of whether the infimum of the Kottman constants taken on renormings of $Z_2$ is equal to $\sqrt{2}$ (\cite[Problem 2]{cgp}) emerges from there.\medskip

Thus, to study the Kottman constant of a twisted sum $X$ with no specific norm, it is natural to consider the isomorphic Kottman constant, $\tilde K(X)$, as introduced in \cite{cgp}; it is the infimum of the Kottman constants of all renormings of $X$. One can analogously define the isomorphic symmetric or finite Kottman constants: ${\widetilde K_s}(X)$ and $\widetilde K_f(X)$.
It is clear that the three parameters $\tilde K(\cdot)$, ${\widetilde K_s}(\cdot)$, and $\widetilde K_f(\cdot)$ are continuous with respect to the Kadets metric too.\smallskip

As for the interpolation issues, if the couple $(X_0, X_1)$ is replaced by some isomorphic copy  $(\tilde X_0, \tilde X_1)$, then one gets an interpolation space $\tilde X_\theta$ isomorphic to $X_\theta$. Therefore, also the three parameters $\tilde K(\cdot)$, ${\widetilde K_s}(\cdot)$,
and $\widetilde K_f(\cdot)$ are continuous with respect to the interpolation parameter and verify moreover the interpolation
inequality. In particular, one also obtains the inequality $2 \leqslant \tilde K_s(X)\cdot \tilde K_s(X^*)$.

In this section we solve problems (1, 2) posed in \cite{cgp}. Problem (1) was to establish the equality
$\tilde K(X) = \max \{\tilde K(Y), \tilde K(Z)\}$, when $X$ is a twisted sum of $Y$ and $Z$. We then prove the following fact.

\begin{prop}\label{sol(1)}
Let $0 \to Y \to X \to Z \to 0$ be an exact sequence of Banach spaces. Then $$\tilde K(X) = \max \{\tilde K(Y), \tilde K(Z)\}.$$
Analogous inequalities hold for $\tilde K_s(\cdot)$ and $\tilde K_f(\cdot)$ too.\end{prop}
\begin{proof} Again, there is no loss of generality in assuming that $\tilde K(X) = K(\tilde X)$. Thus

$$|\tilde K(A) - \tilde K(B)| = |K(\tilde A) - K(\tilde B)|\leqslant 2\cdot g(\tilde A, \tilde B).$$

The space $Y\oplus_1 Z$ is a subspace of $X\oplus_1 Z$. Let $q:X\to Z$ denote the quotient map. For each positive $\varepsilon$, the subspace $X_\varepsilon= \{(\varepsilon x, qx)): x\in X\}$ of $X\oplus_1 Z$ is isomorphic to $X$.
Both equalities follow from $\lim_{\varepsilon\to 0} g(X_\varepsilon, Y\oplus_1 Z)=0$,
which is a consequence of \cite[Lemma 5.9]{ost}.
\end{proof}

Problem (2) was to show that the isomorphic Kottman constant of $Z_2$ is $\sqrt{2}$. Indeed, we prove the following identity.

\begin{cor} If $X$ is a twisted Hilbert space then $\tilde K(X)=\tilde K_s(X)=\tilde K_f(X)=\sqrt{2}$.
\end{cor}


Since we know that $\tilde K(Z_2)=\sqrt{2}$ and since every Banach space $X$ admits a renorming $\tilde X$ so that $K(\tilde X)=2$ \cite{vandulst}, it is natural to ask for renormings that reduce the Kottman constant, a topic that has not been studied so far.\medskip

A renorming that reduces the Kottman constant for $Z_2$ can be made explicit because this space may be
represented as the derived space in an interpolation schema as follows:
Let $(X_0,X_1)$ be an interpolation couple. We set $\Sigma =X_0+X_1$ and define $\mathscr{C}(X_0, X_1)$
to be the \emph{Calderon space} associated to $\Sigma$. We then consider a bounded homogeneous
selection $B\colon X_\theta \to \mathscr{C}$ for the evaluation map $\delta_\theta$.

The space\;
$d_{\delta_\theta' B}X_\theta = \{(y,z)\in  \Sigma  \times X_\theta\, :\, y - \delta_\theta' B z\in X_\theta\}$,\;
endowed with the quasi-norm $$\|(y,z)\| = \|y -\delta_\theta' B z\|_{X_\theta} +\|z\|_{X_\theta},$$ is a twisted
sum of $X_\theta$ with itself since there is a natural exact sequence
$$\begin{CD}
0 @>>> X_\theta @>>> d_{\delta_\theta' B}X_\theta  @>>> X_\theta @>>>0 \end{CD}$$
with inclusion being the map $x\to (x, 0)$ and the quotient map given by $(y,x)\to x$.
If $\delta_\theta'\colon \mathscr C\to \Sigma$ denotes the evaluation of the derivative at $\theta$, the map
$\Omega_\theta = \delta_\theta' B$ is called the \emph{associated derivation}. Two different homogeneous bounded selectors $B$ and $V$ for $\delta_\theta$ may yield different derivations, however
their difference is a bounded map $\delta_\theta' B - \delta_\theta' V\colon X_\theta \to X_\theta$, and consequently
the spaces $d_{\delta_\theta' B}X_\theta$ and $d_{\delta_\theta' V}X_\theta$ are isomorphic.
The Banach space $d_{\delta_\theta' B}X_\theta$ is isomorphic to the so-called derived space
$dX_z =\{(f'(z), f(z))\colon f\in \mathscr C\}$, endowed with the natural quotient norm.

\begin{lemma}\label{kintermax} $K(dX_\theta) \leqslant \max \{K(X_0), K(X_1)\}.$
\end{lemma}
\begin{proof}
Pick a sequence $(z_n)_{n=1}^\infty$ in the unit ball of $dX_\theta$ and for each $z_n$ take an $\e$-extremal $f_n$; \emph{i.e.}, $f_n\in \mathscr C$ with $f_n(\theta)=z_n$ and $\|f_n\|\leqslant\|z_n\|+\e$.
In order to estimate $\|z_n - z_m\|$, we have to estimate the norm $\|g\|$ of an extremal $g$; \emph{i.e.}, a function $g\in \mathscr C$
so that $g(\theta)=z_n - z_m$ and minimal $\|g\|$. 
Given $\varepsilon$ one has:
$$
\|f_n(it) - f_m(it)\|_{X_0} \leqslant K(X_0) (1+\varepsilon)\quad 
\textrm{ and }\quad
\|f_n(1+it) - f_m(1+it)\|_{X_1} \leqslant K(X_1) (1+ \varepsilon),
$$
which yields $\|f_n - f_m\|\leqslant \max \{K(X_0), K(X_1)\} (1+ \varepsilon)$.
\end{proof}

\begin{prop}
$\tilde K(dX_\theta) = \tilde K(X_\theta)$.
\end{prop}

\begin{proof}
Pick $s\leqslant \theta \leqslant t$. By the reiteration formula \cite{belo}, one has
$X_\theta = \left( (X_0, X_1)_t, (X_0, X_1)_s\right)_\nu$ and thus
$K(dX_\theta)\leqslant \max\lbrace K\left((X_0, X_1)_t\right), K\left((X_0, X_1)_s\right)\rbrace$ by Lemma \ref{kintermax}.
Here $X_\theta$ carries the norm derived from the new interpolation couple (which is the same it was before) as well
as $d(X_\theta)$ (which is not). By continuity of $K(\cdot)$ with respect to the interpolation parameter one gets
$\tilde K(dX_\theta) \leqslant \lim_{t\to \theta, s\to \theta} \max \{K(X_t), K(X_s)\} = K(X_\theta)$.
Since $\tilde K(X_\theta) \leqslant \tilde K(dX_\theta)$, the equality is then clear.
\end{proof}

Let us put the above considerations into a more general context. Let $0\to Y \to X \to Z\to 0$ be an exact sequence of Banach spaces. Denoting by 
$\varepsilon\colon Z\to Z$ the map ``multiplication by $\varepsilon$'', we may form a commutative diagram
$$\begin{CD}\label{diagram}
0 @>>> Y @>>> X @>q>> Z @>>>0\\
&&@|@AA{\underline {\e}}A @AA\e A\\
0 @>>> Y @>>> \PB_\e @>>> Z @>>>0\end{CD}$$
(Here $\PB_\e = \{(x,z')\colon qx=\e z'\}$ is considered a subspace of $X\oplus_\infty Z$.)
The map $\underline {\e}$ is an isomorphism that produces a renorming $\tilde X$ such that $K(\tilde X)\leqslant \max \{ K(Y), K(Z)\} + \e$:
Indeed,
$$\PB = \{(x,z')\colon qx=\e z'\} = \{((y,z), z')\colon z= \e z'\} =\{(\varepsilon (y,z), z)\colon (y,z)\in X\} = X_\e$$
algebraically.
While $\PB$ is endowed with the norm inherited from $X\oplus_\infty Z$, the space $X_\e$ inherits the norm from $X\oplus_1 Z$. The arguments of Ostrovskii \cite{ost} to show that $g(X_\varepsilon, Y\oplus_1 Z)\leqslant \varepsilon$ may be used verbatim to show that also $g(\PB, Y\oplus_\infty Z)\leqslant \varepsilon$.
This means that a certain renorming of $X$ has the Kottman constant at most equal to $ \max\lbrace K(Y), K(Z)\rbrace+ \e$.
The diagram above shows that this renorming can be obtained as follows. We pick a quasi-linear map $\Omega$ associated  to the upper exact sequence in (\ref{diagram}).
The quasi-linear map associated to the lower sequence in (\ref{diagram}) is then $\e\Omega$.
Thus, if the space $X$ has as associated quasi-norm $\|(y,x)\|= \|y - \Omega x\| + \|x\|$ then the isomorphic copy 
below $\PB_\e$ has as associated quasi-norm $\|(y,x)\|= \|y - \e\Omega x\| + \|x\|$.
This is what we did in the interpolation situation: if $\Omega_\theta$ is the quasi-linear map associated to the couple
$(X_0, X_1)$ at $\theta$, then the quasi-linear map associated to the couple $(X_t, X_s)$ at $\theta$ is $(s-t)\Omega$.




\section{The disjoint Kottman constant}

One of the surprising things regarding the Kottman constant is that $K(\cdot)$ is not continuous on the scale of $\ell_p$ spaces as $p\to \infty$, while $K(L_p)$ is continuous.
Recall that $K(\ell_p)=2^{1/p}$ for $1\leqslant p<\infty$, whilst $K(\ell_\infty)=2$. On the other hand $K(L_p)=2^{1/p}$ for $1\leqslant p\leqslant 2$ and $K(L_p)=2^{1/p^*}$ for $2\leqslant p \leqslant \infty$.
To clarify this situation we introduce the disjoint Kottman constant on Banach lattices.

\begin{definition}
Let $X$ be a Banach lattice. The \emph{disjoint Kottman constant}, $K^\perp(X)$, is defined as the
supremum of the separation of disjointly supported sequences in the unit ball of $X$.
\end{definition}

The symmetric $K^\perp_s(\cdot)$ and finite $K^\perp_f(\cdot)$ disjoint Kottman constants are analogously defined. The first surprise comes when one realises that the Elton--Odell theorem does not apply here since $K^\perp(c_0)=1=K^\perp(\ell_\infty)= K^\perp(L_\infty)$.

On the other hand, $K^\perp(\cdot)$ is continuous on the whole scale of $\ell_p$ spaces. 
It is also continuous on the scale of $L_p$ spaces since $K^\perp(L_p)=K^\perp(\ell_p)$. 
The disjoint Kottman constant behaves even better in regard to interpolation. 

\begin{prop} 
Let $(X_0, X_1)$ be an interpolation couple of K\"othe spaces. Then
$$K^\perp(X_\theta)\leqslant K^\perp(X_0)^{1-\theta}K^\perp (X_1)^\theta$$
\end{prop}
\begin{proof}
It is well-known that complex interpolation for K\"othe spaces is plain factorisation \cite{kaltmon}: thus, let us choose a disjointly supported sequence of norm-one vectors $(x_n)_{n=1}^\infty$ so that
$\|x_n- x_m\|\geqslant  K^\perp(X_\theta)-\e$ and observe that its almost optimal factorisation
$x_n = y_n^{1-\theta}z_n^\theta$ is also formed by disjointly supported elements:
Thus $x_n - x_m = (y_n- y_m)^{1-\theta}(z_n - z_m)^\theta$, which implies that
$$
K^\perp(X_\theta)-\e \leqslant \|x_n - x_m\| \leqslant
\|y_n- y_m\|_0^{1-\theta}\|z_n - z_m\|^\theta \leqslant K^\perp(X_0)^{1-\theta}K^\perp(X_1)^\theta.
$$
\end{proof}

Note that, unlike in Theorem \ref{interpol}, the interpolation inequality is valid fror $a=0$ and $b=1$. 

The factorisation/interpolation $X_\theta = X_0^{1-\theta}X_1^{\theta}$ may be generalized for families of spaces; according to \cite[Theorem 3.3]{kaltdiff}, Kalton credits Hernandez \cite{Hernandez} for this construction. Given K\"othe function spaces $X_1, \ldots, X_n$ and positive
numbers $a_1, ..., a_n$, we define
$$\prod\limits_{j=1}^n X_j^{a_j} = \{f \in L_0 \colon \left |f\right| \leqslant
\prod\limits_{j=1}^n \left|f_j\right|^{a_j}, f_j \in X_j\}$$
endowed with the norm
$\|f\|_{\prod} = \inf\{\prod_{j=1}^n \|f_j\|_{X(j)}^{a_j}\colon f_j\in X_j, |f| \leqslant \prod_{j=1}^n
\left|f_j\right|^{a_j}, j = 0, 1, 2\ldots\}.$
Then, given disjoint arcs $A_1,\ldots, A_n$ so that $\mathbb T= \cup_{j=1}^n A_j$, if we set $X_\omega = X_j$
on $\omega \in A_j$, $j = 1,\ldots, n$
and if $\mu_{z_0}$ denotes the harmonic measure on $\mathbb T$ with respect to $z_0$, then under minimal
conditions to perform complex interpolation for a finite family of spaces one has
$$
X_{z_0} = \prod\limits_{j=1}^n X_j^{\mu_{z_0}(A_j)}.
$$
Consequently, under the same conditions,
$$K^\perp\left( X_{z_0}\right) \leqslant \prod\limits_{j=1}^n K^\perp\left( X_j\right)^{\mu_{z_0}(A_j)}.$$



Given a K\"othe space $\lambda$ with base measure space $(S, \mu)$, its K\"othe dual is defined as
$$\lambda^\times = \{ f\in L_0(\mu)\colon \Big|\int\limits_S f(s)g(s)\,\mu({\rm d}s)\Big| <\infty\;\; (g\in \lambda)\}.$$
Contrary to the standard duality, one has $\ell_\infty^\times = \ell_1$.
Let us record the following observation on the disjoint Kottman constant and K\"othe duality.


\begin{cor}
$2\leqslant K^\perp_s(\lambda)\cdot K^\perp_s(\lambda^{\times}) \leqslant K^\perp(\lambda)\cdot K^\perp(\lambda^{\times})$.
\end{cor}
Nevertheless, it may still happen that $K^\perp(\lambda)\neq K^\perp(\lambda^{\times \times})$.

\begin{example}
Let us consider the Banach lattice $X=\big( \bigoplus_{n\in \N} \ell_1^n\big)_{c_0}$
%
with the standard discrete K\"othe-space structure.
Then, $X^{\times\times}=X^{**}= \big( \bigoplus_{n\in \N} \ell_1^n\big)_{\ell_\infty}$.

Nevertheless, there exist isometric lattice embeddings $\ell_1\to X^{**}$; for example, the map defined by
$$
(\xi_k)_{k=1}^\infty \mapsto \big( \xi_1, (\xi_1, \xi_2), (\xi_1, \xi_2, \xi_3), \ldots \big).
$$
is such an embedding. Thus $1=K^\perp(X) \neq K^\perp(X^{\times\times}) = K^\perp(\ell_1) =2$.
\end{example}

\section{James' and Whitley's thickness constants}

Whitley introduced in \cite{W} the \emph{thickness} constant $T(\cdot)$ as follows:
$$T(X)= \inf\big\{\varepsilon>0\colon \, \mathrm{there \; exists \; an\;\;}
\varepsilon\mathrm{-net} \;\; F\subset S_X \mathrm{\; for} \;
S_X\}.$$ 
See equivalent formulations in \cite[Prop. 3.4]{MP1} and \cite[Lemma 1]{castpapiH}. One has the following continuity result.

\begin{prop}\label{withleycoka}
The thickness constant is continuous with respect to the Kadets metric. Precisely
$$|T(X) - T(Y)|\leqslant 8 \cdot d_{\rm K}(X, Y).$$
\end{prop}
\begin{proof}
It is clearly enough to show that $|T(M) - T(L)|\leqslant 4\cdot g(M,L)$ for a pair of given subspaces $M,L$ of a Banach space $Z$.
Let us assume for the sake of simplicity that the parameters are attained. Thus, there exist elements $m_1, \dots, m_n \in S_M$ that form a $T(M))$-net for $S_M$. We may then find points $l_i \in L$ for which $\|m_i - l_i\|\leqslant g(M,L)$.
Therefore $1-g(M,L)\leqslant \|l_i\|\leqslant 1+g(M,L)$. Let us consider the points $l_i'= \frac{l_i}{\|l_i\|} \in S_L$. One has
$$\|l_i - l_i'\| = \left\|l_i - \frac{l_i}{\|l_i\|}\right\| =  \|l_i\| -1 \leqslant g(M,L). $$
We show that the points $l_1', \dots, l_n'$ form a $5g(M,N)$-net for $S_L$. Indeed, we pick $l\in S_L$ and get $m_l\in M$ such that
$\|l- m_l\|\leqslant g(M,L)$ and thus $1-g(M,L)\leqslant \|m_l\|\leqslant 1+g(M,L)$.
If $m_l'= \frac{m_l}{\|m_l\|}$ there must be an index $i$ such that $\|m_l' - m_i\|\leqslant T(M)$. Therefore
\begin{eqnarray*}
\|l - l_i'\| &\leqslant& \|l - m_l\| + \|m_l - m_l'\| + \|m_l' - m_i\|+ \|m_i - l_i\|+ \|l_i - l_i'\|\\
&\leqslant&  g(M,L) +  g(M,L) + T(M) +  g(M,L) +  g(M,L).
\end{eqnarray*}

Thus $T(L)\leqslant T(M) + 4g(M,L)$. Exchanging the r\^{o}les of $M$ and $L$, one obtains the estimate $T(M)\leqslant T(L)+4\cdot g(M,L)$,
and consequently
$$|T(M) - T(L)|\leqslant 4\cdot g(M,L).$$
The estimate $|T(X) - T(Y)|\leqslant 8\cdot d_{\rm K}(X,Y)$ then follows.
\end{proof}

It is immediate that $T(\cdot)$ is continuous with respect to the interpolation parameter; precisely
$$|T(X_\theta) - T(X_\eta)|\leqslant 16\left|\frac{\sin \left(\pi(t-s)/2\right)}{\sin \left(\pi(t+s)/2\right)}\right|.$$
This suggests the problem of whether there is an interpolation inequality of the form
$$T(X_\theta)\leqslant T(X_0)^{1-\theta}\cdot T(X_1)^\theta.$$

The behaviour of $T(\cdot)$ is quite analogous to  the behaviour of isomorphic Kottman constants, as we have the following proposition.

\begin{prop}
For every space $X$, $1=\inf T(\tilde X) \leqslant \sup T(\tilde X)=2$
\end{prop}
\begin{proof}
In \cite[Theorem 2 (3)]{castpapiH} it was proved that $T(X\oplus_\infty Y)=\min \lbrace T(X), T(Y)\rbrace$.
Take a~hyperplane $H$ of $X$ so that $X\cong H\oplus \R$.
Since $g(X_\varepsilon, H\oplus_1 \R)\leqslant \varepsilon$ it follows from Proposition \ref{withleycoka}
that $\inf T(\tilde X)\leqslant T(\R)=1$.
Also, \cite[Theorem 2 (2)]{castpapiH} demonstrates that $T(X\oplus_1 Y)= 2$. Since
$g(X_\varepsilon, H\oplus_1 \R)\leqslant \varepsilon$, it follows from Proposition \ref{withleycoka}
that $\sup T(\tilde X)=2$.
\end{proof}

The proposition is intriguing because a Hilbert space---actually any Banach space not containing $\ell_1$---can not be renormed to have $T=2$, even if $\sup T(\tilde \ell_2)=2$.
This could be relevant for the problem of whether $\tilde K(X)=1$ is possible (even when $K(\tilde X)=1$ is not).
There is a connection between Whitley and Kottman constants, namely $$K^s(X)\geqslant  T(X),$$ from which one may directly
obtain the result from \cite{H-kania-R} saying that $\sup K^s(\tilde X)=2$ for every infinite-dimensional Banach space.\medskip

Let $X$ be a Banach space and let $m(x,y) = \min \{\|x-y\|, \|x+y\|\}$ ($x,y\in X$). The \emph{James constant} of $X$ as defined in \cite{P} is the number
${\rm Jm}(X)= \sup_{x \in S} \sup_{y \in S} m(x,y)$. 
\begin{lemma}
The James constant ${\rm Jm}(\cdot)$ is continuous with respect to the Kadets metric. More precisely
$$|{\rm Jm}(X) - {\rm Jm}(Y)| \leqslant 4\cdot d_{\rm K}(X, Y).$$
\end{lemma}
\begin{proof} Pick $x_1, x_2\in S_X$ such that $\|x_1-x_2\|\geqslant {\rm Jm}(X)$ and $\|x_1+x_2\|\geqslant {\rm Jm}(X)$.
Then we may pick $y_1\in Y$ such that $\|x_1-y_1\|\leqslant g(X,Y)$ and $y_2\in Y$ such that $\|x_2-y_2\|\leqslant g(X,Y)$.
One has $\|y_1\| \geqslant  \|x_1\|-\|y_1 - x_1\|\geqslant  1 - g(X,Y)$ and $\|y_2\| \geqslant  1 - g(X,Y)$ as well.
Set $y_1'=\frac{y_1}{\|y_1\|}$ and $y_2'=\frac{y_2}{\|y_2\|}$.
One has $\|y_1 - y_1'\|\leqslant g(X,Y)$ and $\|y_2 - y_2'\|\leqslant g(X,Y)$. Therefore
$$\|y_1' - y_2'\|\geqslant  \|y_1 - y_2\| - 2\cdot g(X,Y) \geqslant  \|x_1 - x_2\| - 4\cdot g(X,Y) \geqslant {\rm Jm}(X) - 4\cdot g(X,Y)$$
and
$$\|y_1' + y_2'\|\geqslant  \|y_1 + y_2\| - 2\cdot g(X,Y) \geqslant  \|x_1 + x_2\| - 4\cdot g(X,Y) \geqslant {\rm Jm}(X) - 4\cdot g(X,Y)$$
Thus ${\rm Jm}(Y)\geqslant  {\rm Jm}(X) - 4\cdot g(X, Y)$.
Interchanging the r\^{o}les of $Y$ and $X$ one readily gets the desired inequality ${\rm Jm}(X)\geqslant  {\rm Jm}(Y) - 4\cdot g(X, Y)$.
\end{proof}


\begin{remark}
Let $M(x,y) = \max \{\|x-y\|, \|x+y\|\}$ ($x,y\in X$) and set $$g(X)= \inf_{x \in S} \inf_{y \in S} M(x,y).$$ It was shown in \cite{castpapi} that $g(\cdot) \leqslant T(\cdot) \leqslant K_s(\cdot) \leqslant {\rm Jm}(\cdot)$
and $g(\cdot)\cdot {\rm Jm}(\cdot)=2$.
Thus, since ${\rm Jm}(\cdot)$ is continuous with respect to the Kadets metric, so is $g(\cdot)$.
\end{remark}


\end{document}